\documentclass{amsart}
\usepackage[margin=1in]{geometry}
\usepackage{amsmath, amsfonts, amssymb, amsthm}
\usepackage{graphicx}
\usepackage{hyperref, url}
\usepackage{color}
\usepackage{dsfont}

\newtheorem{thm}{Theorem}[section]
\newtheorem{prop}[thm]{Proposition}
\newtheorem{cor}[thm]{Corollary}
\newtheorem{ass}[thm]{Assumption}
\newtheorem{lemma}[thm]{Lemma}
\newtheorem{defn}[thm]{Definition}

\newtheorem{preremark}[thm]{Remark}
\newenvironment{remark}{\begin{preremark}\rm}{\medskip \end{preremark}}
\numberwithin{equation}{section}

\newcommand{\one}{\mathds{1}}

\newcommand{\R}{\mathbb R}
\newcommand{\eps}{\varepsilon}

\newcommand{\grad} {\nabla}

\newcommand{\dd} {\; \mathrm{d}}
\newcommand{\dist} {\mathrm{dist}}

\DeclareMathOperator*{\osc}{osc}

\DeclareMathOperator{\dv}{div}

\DeclareMathOperator*{\esssup}{ess-sup}
\DeclareMathOperator*{\essinf}{ess-inf}
\newcommand{\hull}[1]{\text{Convex Env.}\left(#1\right)}

\title{Oscillation properties of scalar conservation laws}
\date{\today}
\author{Luis Silvestre}
\thanks{LS was partially supported by NSF grants DMS-1254332 and DMS-1362525.}

\begin{document}
\begin{abstract}
We obtain several new regularity results for solutions of scalar conservation laws satisfying the genuine nonlinearity condition. We prove that the solutions are continuous outside of the jump set, which is codimension one rectifiable. We show that the entropy dissipation vanishes away from the closure of the jump set. We prove that the solution decays algebraically in $L^\infty$ as $t \to \infty$ and we compute the presumably optimal decay rate. All these results are based on a local oscillation estimate which is obtained properly adapting some ideas of De Giorgi from the context of elliptic equations.
\end{abstract}

\maketitle

\section{Introduction}

In this work, we study entropy solutions to scalar conservation laws
\begin{equation}\label{e-conservationlaw}
 u_t + \dv A(u) = 0 \text{ in } (0,\infty) \times \R^d.
\end{equation}
Here $A: \R \to \R^d$ is a given function. We use the standard notation $a(v) = A'(v)$.

For some regularity considerations, it is convenient to study time independent conservation law equations of the form
\begin{equation} \label{e:conservatiolaw-no-t}
 \dv [ A(u) ] = a(u) \cdot \nabla u = 0.
\end{equation}
This formulation is not less general than \eqref{e-conservationlaw}, since we can consider $u$ as a function of $(t,x_1,\dots,x_d) \in \R^{d+1}$ with $\tilde a(t,x_1,\dots,x_d) = (1,a(t,x))$ so that a solution to \eqref{e-conservationlaw} is also a solution to \eqref{e:conservatiolaw-no-t} with $\tilde a$ instead of $a$. Conversely, a solution of an equation in the form \eqref{e:conservatiolaw-no-t} is obviously also a solution of an equation in the form \eqref{e-conservationlaw} which is constant in $t$.

Whenever possible, we will state our results in terms of the equation \eqref{e:conservatiolaw-no-t}. We do this only to make the formulas cleaner, since the equation is written with one fewer term.  The function $a$ in \eqref{e:conservatiolaw-no-t} will satisfy the usual \emph{genuine nonlinearity} condition.

\begin{ass} \label{a:genuine-nonlinearity}
The function $a$ is $C^1$ and there exists $\alpha \in (0,1]$ and $C>0$, so that for every $\xi \in \R^d$ with $|\xi|=1$, and any $\delta > 0$, we have
\[ | \{ v \in I : |a(v) \cdot \xi| < \delta \} | \leq C \delta^\alpha.\]
Here $I$ is a closed interval that contains the image of the function $u$.
\end{ass}

Under this assumption, it is well known that conservation laws enjoy striking regularization properties. This was first obtained in \cite{lions1994kinetic} using the kinetic formulation of conservation laws. 

Solutions to \eqref{e-conservationlaw} whose initial data belongs to $BV(\R^d)$, stay in $BV$ for positive time. In \cite{de2003structure}, the authors study non-$BV$ solutions and prove that they have a similar structure as $BV$ functions in the following sense.
\begin{itemize}
\item There is a jump set $J$ which is codimension one rectifiable.
\item The solution $u$ has vanishing mean oscillation at every point outside of $J$.
\item The function $u$ has left and right traces on $J$ in the sense that for almost all point $x_0 \in J$, blow up limits centered at $x_0$ converge in $L_{loc}^1$ to single shock solutions.
\end{itemize}

The jump set $J$ is defined explicitly in terms of the entropy dissipation measure. We recall its definition in \eqref{e:shock-set}.

In this paper we obtain several new regularity properties of conservation laws satisfying the genuine nonlinearity condition. The first of our main results tells us that the solution must be continuous outside of the jump set.

\begin{thm} \label{t:intro-continuity}
Let $u$ be an entropy solution of the equation \eqref{e:conservatiolaw-no-t} satisfying Assumption \ref{a:genuine-nonlinearity}. Then $u$ is continuous outside of the jump set $J$ (or, more properly, a.e. equal to a continuous function).
\end{thm}

This result improves the VMO condition obtained in \cite{de2003structure}. Note that vanishing mean oscillation is weaker than being a Lebesgue point. The fact that every point outside of $J$ is a Lebesgue point is proposed as an open problem in \cite{de2003structure} and \cite{crippa2008regularizing}. It has been established only in one space dimension in \cite{deLellisRiviere2003}. Here we go a step further by proving that $u$ is in fact continuous outside of $J$ in any dimension. Our result was conjectured in \cite{deLellisRiviere2003} (see Remark 1.3 there). It is new even in the context of $BV$ solutions.

Note that Theorem \ref{t:intro-continuity} holds at every point in $x \in \Omega \setminus J$. Even if $x$ is an accumulation point of $J$, we obtain
\[ \lim_{r \to 0} \left( \esssup_{B_r(x)} u - \essinf_{B_r(x)} u \right) = 0.\]
Another interpretation of Theorem \ref{t:intro-continuity} is that there are a lower semicontinuous function $\underline u$ and an upper semicontinuous function $\overline u$ such that $\underline u = u = \overline u$ almost everywhere, and $\overline u = \underline u$ in $\Omega \setminus J$.

Another interesting conjecture concerns the concentration of the entropy dissipation measure on the jump set $J$. It is related to the first open question in \cite{perthame2002kinetic} (section 1.13). It is also mentioned in \cite{de2003structure}, \cite{crippa2008regularizing} and \cite{deLellisRiviere2003}. It is known to hold in the context of $BV$ solutions because of Volpert chain rule (see for example \cite{ambrosio2007}) and also in general for one dimensional problems (see \cite{deLellisRiviere2003}, \cite{bianchini2016} and \cite{bianchini2016structure}). Here, we  prove that for any $L^\infty$ entropy solution of \eqref{e:conservatiolaw-no-t}, in any dimension, satisfying Assumption \ref{a:genuine-nonlinearity}, there is no entropy dissipation outside of the closure $\overline J$. That is our second main result.

\begin{thm} \label{t:intro-no-entropy}
Let $u: \Omega \to \R$ be an entropy solution of \eqref{e:conservatiolaw-no-t} satisfying Assumption \ref{a:genuine-nonlinearity}. Let $\mu$ be its kinetic entropy dissipation measure and $J$ be its jump set. Then $\mu( (\Omega \setminus \overline J) \times \R) = 0$.
\end{thm}

Here, $\mu$ is the measure which appears in the right hand side of the usual entropy formulation of the conservation law introduced in \cite{lions1994kinetic}. The jump set $J$ was introduced in \cite{de2003optimality} and is recalled in \eqref{e:shock-set}. The domain $\Omega$ can be any open set.

The conjecture is not fully resolved since the jump set $J$ might be strictly smaller than its closure $\overline J$ in some pathological cases.

Our third main result concerns the decay of the solution to \eqref{e-conservationlaw} in $L^\infty$ as $t \to \infty$. In order to obtain a sharp exponent, we need a more precise version of Assumption \ref{a:genuine-nonlinearity}. 

\begin{ass} \label{a:hormander}
Let the function $u$ which solves \eqref{e:conservatiolaw-no-t} take values in a bounded closed interval $I \subset \R$. We assume that there is a positive integer $m$ so that the function $a$ is of class $C^m(I)$ and for every $v \in I$, the vectors $\{a(v), a'(v), a''(v), \dots, a^{(m)}(v)\}$ span $R^d$. 
\end{ass}

It is easy to see that when $a$ is sufficiently smooth, Assumption \ref{a:genuine-nonlinearity} and Assumption \ref{a:hormander} are actually equivalent and $m =  1/\alpha$. This is explained in detail in subsection \ref{ss:assumptions}. Note that we stated these assumptions for the function $a$ in \eqref{e:conservatiolaw-no-t}. For $t$-dependent equations as in \eqref{e-conservationlaw}, we would require $(1,a(v))$ to satisfy these assumptions instead of $a(v)$.

The advantage of Assumption \ref{a:hormander} over Assumption \ref{a:genuine-nonlinearity} is that we can scale the solution in a precise way preserving this condition (see Section \ref{s:scaling}).

\begin{thm} \label{t:intro-decay}
Let $u$ be an entropy solution to \eqref{e-conservationlaw}. We assume that the initial data $u_0 := u(0,\cdot) \in L^1(\R^d) \cap L^\infty(\R^d)$ and also Assumption \ref{a:hormander} holds for $(1,a(v))$. Let
\[ \gamma_0 := \left( 1 + \frac{   d (2m-d+1)  }2 \right)^{-1}. \]
 Then, for any $\gamma \in (0, \gamma_0)$, we have
\[ |u(t,x)| \leq C \|u_0\|_{L^1}^\gamma t^{-d \gamma} \qquad \text{a.e.},\]
where $C$ is a constant that depends on $\|u_0\|_{L^\infty}$ and the function $a$. 
\end{thm}

The exponent $t^{-\gamma_0}$ is the optimal decay in $L^\infty$ for solutions to conservation law equations at least in the case of one space dimension (see Remark \ref{r:optimal-decay}). It is well known that entropy solutions to the Burgers equation $u_t + u \, u_x = 0$ satisfy the estimate
\[ |u(t,x)| \leq C t^{-1/2} \|u_0\|_{L^1}^{1/2}.\]
In this case $m=d=1$ and $\gamma_0 = 1/2$.

In some cases, we can compute explicitly how the constant $C$ in Theorem \ref{t:intro-decay} depends on $\|u_0\|_{L^\infty}$. See Remark \ref{r:decay-with-Linfty-norm}.

As far as we know, previous results concerning the decay rate of solutions in $L^\infty$ were restricted to some particular 1D models (see for example the classical work \cite{lax1957} for $A$ convex). There are some interesting recent results about the decay in $L^1$ (by interpolation also its $L^p$ norm for $1\leq p<\infty$) in the periodic setting, which is of rather different nature (See \cite{chen1999}, \cite{chenperthame2009}, \cite{Debussche2009}, \cite{panov2013}, \cite{dafermos2013}, \cite{gess2017long}). In fact, combining the results in the literature with our theorem \ref{t:degiorgi} tells us that periodic solutions converge to their averages in $L^\infty$ as well.

We conclude our paper with a discussion in Proposition \ref{p:traces} of some consequences of our results regarding the traces of the function $u$ on the jump set $J$.

The key for the proofs of all these results is the oscillation estimates given by Theorem \ref{t:degiorgi}. This theorem gives us an estimate for the pointwise oscillation of the function $u$ in terms of its averaged oscillation. Its proof uses some ideas originally deveoped by De Giorgi for elliptic equations in \cite{e1957sulla}.

The famous proof of De Giorgi's theorem for elliptic equations has two steps. The first step consist in obtaining a local estimate in $L^\infty$ in terms of the $L^2$ norm of the solution in a larger ball. In the proof of this first step, a local gain of integrability for truncations is obtained using the energy dissipation inequality and this is iterated to finally get the estimate in $L^\infty$. In our proof we set up a similar iteration. We use the gain of integrability given by averaging instead of energy dissipation. The regularization effect of velocity averaging is a powerful tool in the study of kinetic equations and conservation laws. It originated in \cite{golse1988}.

The second step in the proof of De Giogi gives a local improvement of oscillation which leads to the H\"older continuity of solutions to uniformly elliptic equations. That step will not hold in general for conservation laws, since discontinuities do occur. However, outside of the jump set $J$, with the help of the VMO condition obtained in \cite{de2003structure}, we deduce the continuity of the solution to prove Theorem \ref{t:intro-continuity} in section \ref{s:continuity}.

It is not the first time that De Giorgi's method is used outside of the realm of classical elliptic or parabolic equations. Some other previous unorthodox applications of this method are to the surface quasi-geostrophic equation \cite{caffarelli2010drift}, to the Hamilton-Jacobi equation \cite{chan2014giorgi}, to certain one-dimensional active scalar equation \cite{silvestre2016transport}, to kinetic-diffusion equations \cite{golse2016harnack} and to the Boltzmann equation \cite{imbert2016weak}. In each of these applications, the underlying mechanism by which we obtain a local gain of integrability is fundamentally different.

The proof of Theorem \ref{t:intro-no-entropy} uses the continuity of $u$ outside of $J$, which is given in Theorem \ref{t:intro-continuity}. For continuous solutions to conservation laws, we prove that there are well defined characteristic curves that are straight lines. The lack of entropy dissipation follows from this characterization. This idea goes back to the work of C. Dafermos \cite{dafermos2006}. We prove Theorem \ref{t:intro-no-entropy} in Section \ref{s:no-entropy}. 

The decay given in Theorem \ref{t:intro-decay} is obtained by combining the estimate in Theorem \ref{t:degiorgi} with the scaling of the equation explained in Section \ref{s:scaling}. We prove Theorem \ref{t:intro-decay} in section \ref{s:decay}.

Most regularity results based on averaging hold also for \emph{generalized solutions} or \emph{quasi-solutions}. In fact, the regularization results by averaging are optimal within this class (see \cite{de2003optimality}). The structural results in \cite{de2003structure} and \cite{crippa2008regularizing} also hold for this generalized notion of solution. Interestingly, the results that we give here hold for entropy solutions only. Indeed, it is easy to find examples to see that Theorem \ref{t:intro-decay} is not true for generalized solutions. Theorem \ref{t:degiorgi} holds for entropy solutions only. In our proof, this plays a role when considering truncations with Lemma \ref{l:max-of-subsolutions} and estimating the total variation of $\mu_0$ and $\mu_1$ in each iteration using Lemmas \ref{l:total-measure1} and \ref{l:total-measure0}.

\section{Preliminaries}
\label{s:preliminaries}

\subsection{Entropy solutions}

We recall the following standard definition of entropy solutions and subsolutions.

\begin{defn} \label{d:entropysolutions}
For every convex function $\eta : \R \to \R$, let $q: \R \to \R^d$ be a function such that
\[ q_i'(v) = \eta'(v) a_i(v).\]
for $i = 1,\dots,d$.

Let $u : \R^d \to \R$ belong to $L^\infty(\R^d)$. We say $u$ is an entropy solution when the following happens. For every convex function $\eta$, and any smooth, compactly supported, test function $\varphi \geq 0$, we have
\begin{equation} \label{e-entropycondition}
 \int_{\R^n} q(u) \cdot \grad \varphi \dd x \geq 0. 
\end{equation}

When the inequality \eqref{e-entropycondition} holds only for $\eta$ convex and non-decreasing, we say $u$ is an entropy subsolution.
\end{defn}


It is well known that it is enough to consider $\eta(u)$ of the form $(u-\ell)_+$ and $(\ell - u)_+$ with $\ell \in \R$, in order to verify that $u$ is an entropy solution.

For $t$-dependent equations as in \eqref{e-conservationlaw}, it is well known (since \cite{kruvzkov1970first}) that for any initial condition $u_0 \in L^1(\R^d) \cap L^\infty(\R^d)$,  there is a unique entropy solution $u \in C([0,+\infty),L^1(\R^d)) \cap L^\infty([0,+\infty) \times L^\infty(\R^d))$.

A few times in this paper, we will use that the maximum between an entropy subsolution and a constant is also an entropy subsolution, which is a straight forward consequence of the definitions. However, it is also a particular case of the general fact that the maximum between two entropy subsolutions is also an entropy  subsolution. We state and prove that interesting fact as a lemma here. As far as we know, it was first observed in the work of P.L. Lions and P. Souganidis \cite{lions-video}.

\begin{lemma} \label{l:max-of-subsolutions}
Let $u$ and $v$ be two entropy subsolutions of \eqref{e-conservationlaw} in a convex domain $\Omega$. Then $\max(u,v)$ is also an entropy subsolution.
\end{lemma}
\begin{proof}
The proof is based on Krushkov's idea of doubling variables.

We want to verify that the function $w(x) = \max( u(x) , v(x))$ satisfies \eqref{e-entropycondition} for any $\eta$ convex and nondecreasing.

For every fixed $y \in \Omega$, we apply the definition \eqref{e-entropycondition} to $u$ with $\tilde \eta(u) = \eta(\max(u,v(y)))$. Note that this function $\tilde \eta$ is convex and non-decreasing. We observe that $\tilde q(u) = q(\max(u,v(y)))$ satisfies $\tilde q'(u) = \tilde \eta'(u) a(u)$. Thus,
\[ \int_{\Omega} q(\max(u(x),v(y))) \cdot \nabla_x \left[ \varphi\left( \frac{x+y}2 \right) b_\eps\left( x-y \right) \right] \dd x \geq 0.\]
Here $\varphi : \Omega \to \R$ is an arbitrary test function and $b_\eps$ is an approximation of the Dirac mass.

Likewise, for every fixed $x \in \Omega$, we obtain
\[ \int_{\Omega} q(\max(u(x),v(y))) \cdot \nabla_y \left[ \varphi\left( \frac{x+y}2 \right) b_\eps\left( x-y \right) \right] \dd x \geq 0.\]

Integrating the first inequality in $y$, the second in $x$, and adding them, we obtain
\[ \iint_{\Omega \times \Omega} q(\max(u(x),v(y))) \nabla \varphi \left( \frac{x+y}2 \right) b_\eps\left( x-y \right) \dd x \dd y \geq 0.\]
Taking the limit as $\eps \to 0$, we obtain
\[ \int_{\Omega} q(\max(u(x),v(x))) \nabla \varphi (x) \dd x \geq 0.\]
This justifies that $w(x) = \max(u(x),v(x))$ is an entropy subsolution.
\end{proof}

We recall the kinetic formulation of conservation laws given in \cite{lions1994kinetic}. Given a function $u: \Omega \to \R$, we define the function $f : \Omega \times \R \to \{-1,0,1\}$,
\begin{equation} \label{e:kinetic-formulation}
f(x,v) = \begin{cases}
1 &\text{if } 0 < v < u(x), \\
-1  &\text{if } u(x) < v < 0, \\
0 &\text{otherwise.}
\end{cases}
\end{equation}

It was proved in \cite{lions1994kinetic} than $u$ is an entropy solution of \eqref{e:conservatiolaw-no-t} if and only if there is a nonnegative measure $\mu$ in $\Omega \times \R$ such that $a(v) \cdot \nabla_x f = \partial_v \mu$. 

Given a solution $u : \Omega \to \R$ to \eqref{e:conservatiolaw-no-t}, the jump set $J$, introduced in \cite{de2003structure}, is given by
\begin{equation} \label{e:shock-set}
 J :=\left \{ x \in \Omega : \limsup_{r \to 0} \frac{\mu(B_r(x) \times \R)}{r^{d-1}} > 0. \right \}.
\end{equation}

\subsection{Semicontinuous envelopes}

For some of the arguments in this paper, it will be convenient to consider the upper and lower semicontinuous envelopes of a function $u$. They are given in the following definition.

\begin{defn} \label{d:semicontinuous-envelopes}
Let $u \in L^\infty(\Omega)$. We define the lower and upper semicontinuous envelopes of $u$, which we denote $\underline u$ and $\overline u$ respectively, by the following formulas
\[ \underline u(x) = \lim_{r \to 0} \left( \essinf_{\Omega \cap B_r(x)} u \right), \qquad \overline u(x) = \lim_{r \to 0} \left( \esssup_{\Omega \cap B_r(x)} u \right).\]
\end{defn}

It is not difficult to verify that $\underline u$ is lower semicontinuous, $\overline u$ is upper semicontinuous, and $\underline u(x) \leq u(x) \leq \overline u(x)$ almost everywhere. For a general function $u \in L^\infty$, there is no reason why $\overline u(x) = \underline u(x)$ at any point $x$. These equality holds at points where $u$ is continuous. These functions will be meaningful once we establish that $u$ is continuous almost everywhere, after Theorem \ref{t:intro-continuity}. After that, we will deduce that $\underline u(x) = \overline u(x) = u(x)$ almost everywhere. Both functions $\overline u$ and $\underline u$ will be natural representatives of $u$ in the same class in $L^\infty$.

Note also that for any compact set $K \subset \Omega$, we have
\begin{align*} 
 \min \{\underline u(x) : x \in K\} &= \lim_{\delta \to 0} \essinf \{ u(x) : x \in K_\delta\},\\
 \max \{\overline u(x) : x \in K\} &= \lim_{\delta \to 0} \esssup \{ u(x) : x \in K_\delta\},
\end{align*}
where $K_\delta$ is a $\delta$-neighborhood of $K$.

\subsection{Equivalence of Assumptions \ref{a:genuine-nonlinearity} and \ref{a:hormander}}
\label{ss:assumptions}

We discuss how, for a sufficiently smooth function $a$, Assumption \ref{a:genuine-nonlinearity} relates to Assumption \ref{a:hormander}.

Our first proposition shows that Assumption \ref{a:genuine-nonlinearity} implies Assumption \ref{a:hormander} for smooth enough functions.

\begin{prop} \label{p:ass1toass2}
Let $a : I \to \R$ be a function so that Assumption \ref{a:genuine-nonlinearity} holds. Let $m$ be the largest integer smaller or equal to $1/\alpha$, and let us assume that $a \in C^{m+1}(I)$. Then Assumption \ref{a:hormander} holds.
\end{prop}

\begin{proof}
Let $v \in I$ be an arbitrary point and $|\xi|=1$. Since Assumption \ref{a:genuine-nonlinearity} holds, for any $h > 0$, there must be a point $w \in I \cap [v-h,v+h]$ such that $|a(w) \cdot \xi| \gtrsim h^{1/\alpha}$. Therefore, $a^{(j)}(v) \cdot \xi \neq 0$ for some $j \in \{0,1,\dots,m\}$. Since this holds for every unit vector $\xi$, then the vectors $\{ a(v), a'(v), \dots a^{(m)}(v)\}$ generate $\R^d$.

\end{proof}


The opposite implication, from Assumption \ref{a:hormander} to Assumption \ref{a:genuine-nonlinearity} is trickier. We start with a preparatory lemma.

\begin{lemma} \label{l:pre-ass}
Let $I \subset \R$ be an interval. Let $f : I \to \R$. Assume that for some integer $k \geq 0$, we know that $f^{(k)}(v) \geq 1$ for all $v \in I$. Then, for all $\delta \in (0,1)$, 
\[ |\{ v \in I : |f(v)| < \delta^k\}| \leq C \delta.\]
The constant $C$ depends on $k$ only (not on the size of $I$).
\end{lemma}

\begin{proof}
We will prove it by induction in $k$. The case $k=0$ is trivial with $C=0$. Let us assume we have established the result for some given $k$ with a constant $C_k$. Let us prove it for $k+1$.

Since we have $f^{(k+1)}(v) \geq 1$ for all $v \in I$, then $f^{(k)}$ is increasing with derivative always larger or equal to one. The interval $I$ can be decomposed into three subintervals $I + I_1 + I_2 + I_3$, such that
\begin{align*}
f^{(k)}(v) &\leq -\delta \qquad \text{ when } v \in I_1, \\
|f^{(k)}(v)| &\leq \delta \qquad \text{ when } v \in I_2, \\
f^{(k)}(v) &\geq \delta \qquad \text{ when } v \in I_3.
\end{align*}
Moreover, $|I_2| \leq 2 \delta$.

For each subinterval $I_1$ and $I_2$, we apply the inductive hypothesis to $f / (-\delta)$ and $f/\delta$ respectively. We obtain,
\[ |\{ v \in I_j: |f(v)/\delta| <\delta^k\}| \leq C_k \delta \qquad \text{for } j = 1,3.\]
Therefore
\[ |\{ v \in I : |f(v)| < \delta^{k+1}\}| \leq (2C_k + 2) \delta.\]
So, we finish the proof setting $C_{k+1} = (2C_k+2)$.
\end{proof}

\begin{prop} \label{p:ass2toass1}
Let $a$ be a function so that Assumption \ref{a:hormander} holds. Then also Assumption \ref{a:genuine-nonlinearity} holds with $\alpha = 1/m$. The constant $C$ in Assumption \ref{a:genuine-nonlinearity} depends on the modulus of continuity of $a$ in $C^m$ and the positive constant
\[ c_0 := \min_{v\in I,|\xi|=1} \max \{ |\xi \cdot a^{(j)}(v)| : j = 0,1,\dots,m\}.\]
\end{prop}

Note that Assumption \ref{a:hormander} implies $c_0 > 0$.

\begin{proof}
Let $\xi$ be any unit vector. Let $f(v) = 2 \xi \cdot a(v) / c_0$.  Thus, for every $v \in I$, there is some $k \in \{0,1,\dots,m\}$ so that $f^{(k)}(v) \geq 2$.

Using the fact that $f^{(k)}$ is uniformly continuous in $I$, there is an interval $I_v$ around $v$ where $f^{(k)} \geq 1$. Therefore, we can cover $I$ with a finite subcollection of intervals $I_j$, so that $f^{(k_j)} \geq 1$ in $I_j$ with $k_j \in \{0,\dots,m\}$. 
\[ I \subset \bigcup_{j=1}^N I_j.\]
The number of intervals required depends on the size of $I$ and the modulus of continuity of $a$ in $C^m$.

We apply Lemma \ref{l:pre-ass} in each $I_j$ and conclude the proof.
\end{proof}

Propositions \ref{p:ass1toass2} and \ref{p:ass2toass1} above show that Assumptions \ref{a:genuine-nonlinearity} and \ref{a:hormander} are equivalent for smooth functions $a$. In particular, $1/\alpha$ will always be an integer larger or equal than $d-1$. If we consider the function $a(v) = (1,|v|^{1/\alpha})$, it satisfies Assumption \ref{a:genuine-nonlinearity} for any $\alpha \leq 1$. This does not contradict our previous statement, since it is not a smooth function unless $1/\alpha$ is an integer.

\subsection{Notation}
\begin{itemize}
\item $J$ is always the jump set defined in \eqref{e:shock-set}.
\item Given a set $E \subset \R^d$, its topological closure is denoted by $\overline E$.
\item Given an $L^\infty$ function $u$, its upper and lower semicontinuous envelopes, given in Definition \ref{d:semicontinuous-envelopes}, are written $\overline u$ and $\underline u$.
\item Given a measurable set $E \subset \R^d$, we denote its Lebesgue measure by $|E|$.
\item For any set $E \subset \R^d$, we write $\hull{E}$ to denote its convex envelope.
\item Throughout this article, $A: \R \to \R^d$ is the function in the equation \eqref{e-conservationlaw} or \eqref{e:conservatiolaw-no-t}, and $a(v) = A'(v)$.
\item We write $u_0$ to denote the initial condition $u_0 = u(0,\cdot)$ of any solution $u$ of \eqref{e-conservationlaw}.
\item $v_+ = \max(v,0)$, $v_- = \max(-v,0)$.
\item $a^{(k)}$ denotes the $k$th derivative of the function $a$.
\end{itemize}

\section{Scaling}

\label{s:scaling}

In this section we explain a two-parameter family of scalings that leave the equation \eqref{e:conservatiolaw-no-t} invariant.

If $u$ is a solution or a subsolution to \eqref{e:conservatiolaw-no-t}, it is an immediate observation that the scaled function $u_r(x) = u(rx)$ is also a solution for any $r>0$.

The second parameter in our family of scalings will only be used for obtaining optimal decay exponent $\gamma_0$ in Theorem \ref{t:intro-decay}. It is not used for the proofs of Theorems \ref{t:intro-continuity} or \ref{t:intro-no-entropy}.

In order to introduce a second family of scalings, we need the more precise information about the nonlinearity $a(v)$ given in Assumption \ref{a:hormander}. At the point $v=0$, we know that the set of vectors $\{ a(0), a'(0), \dots, a^{(m)}(0)\}$ generate $\R^d$. Let us extract a subset of $d$ linearly independent vectors generating $\R^d$. We call them
$\{a^{(j_1)}(0), a^{(j_1)}(0),\dots,a^{(j_d)}(0)\}$. Moreover, we pick the lexicographically smallest indexes $j_i$ satisfying this condition. Naturally, since we assume that $a \in C^m$, there will be some interval $(-v_0,v_0)$ such that $\{a^{(j_1)}(v), a^{(j_1)}(v),\dots,a^{(j_d)}(v)\}$ is a basis of $\R^d$ for $v \in (-v_0,v_0)$.

For $\lambda \in (0,v_0)$, let $S_\lambda$ be the linear transformation such that for $i=1,\dots,d$,
\[
S_\lambda (a^{(j_i)}(0)) = \lambda^{j_i} a^{(j_i)}(0).
\]

Let us define 
\begin{equation} \label{e:two-parameter-scaling}
u_{r,\lambda}(x) = \lambda^{-1} u(r \, S_\lambda x).
\end{equation}

From a direct computation, we verify that if $u$ satisfies \eqref{e:conservatiolaw-no-t} (or is a subsolution), then $u_{r,\lambda}$ satisfies the same equation with $a(v)$ replaced by
$\tilde a(v) = S_\lambda^{-1} a(\lambda v)$. The choice of $S_\lambda$ was made precisely so that $\tilde a^{(j_i)}(v) = a^{(j_i)}(\lambda v)$. Therefore, if $\lambda < v_0$ and $u(x) \in (-\lambda,\lambda)$, then $u_{r,\lambda}$ will satisfy an equation for which Assumption \ref{a:hormander} with $I = [-1,1]$ is satisfied with a uniform modulus of continuity in $C^m$, and a uniform constant $c_0$ as in Proposition \ref{p:ass2toass1}. In particular, Assumption \ref{a:genuine-nonlinearity} holds uniformly in $\lambda$.

We formulated this scaling procedure and the Assumption \ref{a:hormander} in terms of stationary solutions as in \eqref{e:conservatiolaw-no-t}. For $t$-dependent equations as in \eqref{e-conservationlaw}, $a(v) = (1,\dots)$ and $a^{(j)}(v) = (0,\dots)$ for any $j\geq 1$. The time variable is not affected by the linear transformation $S_\lambda$. For example, for the generalized Burgers equation (as in \cite{crippa2008regularizing}),
\[ u_t + \sum_{k=1}^d u^k \partial_k u = 0,\]
we would have
\[ u_{r,\lambda}(t,x) = \lambda^{-1} u(rt, r\lambda x_1, r \lambda^2 x_2, \dots, r \lambda^{-d} x_d).\]
And $u_{r,\lambda}$ would satisfy the same equation as $u$.

\section{Kinetic characterization of subsolutions}
\label{s:kinetic}

The following proposition extends the kinetic characterization to entropy subsolutions. The proof is very similar to that in \cite{lions1994kinetic}. We have to do some extra work in order to restrict the domain of the measures $\mu_0$ and $\mu_1$.

\begin{prop} \label{p:kinetic-subsolutions}
Let $u$ be a bounded entropy subsolution to the conservation law \eqref{e:conservatiolaw-no-t}. Consider the function $f$ given by \eqref{e:kinetic-formulation}. This function satisfies the kinetic equation
\[ a(v) \cdot \nabla_x f =  \partial_v \mu_0 - \mu_1,\]
where $\mu_0$ and $\mu_1$ are two Radon measures supported in the set $\{ (x,v) : \underline u(x) \leq v \leq \overline u(x)\}$.
\end{prop}

\begin{proof}
The proof essentially follows the steps of the proof in \cite{lions1994kinetic}. We split the domain $B_R$ with an arbitrary open cover (defined below) to refine the information of the support of $\mu_0$ and $\mu_1$.

Let $G_i$ be an open cover of $B_R$. That means that each $G_i$ is an open set and $B_R \subset \bigcup G_i$. Let $\varphi_i$ be a partition of unit associated to this open cover. Let $L_i = \essinf_{G_i \cap B_R} f$ and $U_i = \essinf_{G_i \cap B_R} f$.

Since $u$ is a subsolution $-a(u(x)) \cdot \nabla u(x)$ is a measure. We define $\mu_1$ as $-\sum_i (a(u(x)) \cdot \nabla u(x)) \varphi_i(x) \otimes \delta_{L_i} (v)$. That is, for any $C^1$ function $g: B_R \times \R$, we set
\[ \int_{B_R \times \R} g \dd \mu_1 = \sum_i \int_{B_R} A(u(x)) \cdot \nabla [\varphi_i(x) g(x,L_i)] \dd x.\]

Let $T \in \mathcal{D}'(B_R \times \R)$ be the distribution of order at most one defined by
\[ T = a(v) \cdot \nabla_x f + \mu_1.\]

We define $\mu_0$ as a distribution with the following formula
\[ \langle \mu_0 , g \rangle := - \sum_i \left\langle T , \varphi_i(x) \int_{L_i}^v g(x,w) \dd w \right\rangle, \qquad \text{for every } g \in \mathcal D(B_R \times \R).\]
Note that by the construction of $\mu_1$, we have $\langle T , g \otimes 1 \rangle = 0$ for any $g \in \mathcal{D}(B_R)$. This implies that $\mu_0$ is supported in $\bigcup_i G_i \times [L_i, U_i]$.

In order to show that $\mu_0$ is a measure, we need to verify that $\langle \mu_0 , g \rangle \geq 0$ for any $g(x,v) \geq 0$. It is enough to test with functions of the form $g(x) \psi(v) \geq 0$ (since these generate the whole space $\mathcal D$).

Let $\eta_i$ be the function such that $\eta_i''(v) = \psi(v)$, $\eta_i(L_i) = \eta_i'(L_i) = 0$. Since $\psi \geq 0$, $\eta_i$ is convex and it is non-decreasing in $[L_i,+\infty)$.

We have
\begin{align*} 
 \langle \mu_0 , g \otimes \psi \rangle &= - \sum_i \left\langle T ,  \varphi_i(x) g(x) \int_{L_i}^v \psi(w) \dd w \right\rangle, \\
&= - \sum_i \left\langle T , \varphi_i(x) g(x) \eta_i'(v) \right\rangle, \\
&= \sum_i \left( \iint f(x,v) a(v) \cdot \nabla [\varphi_i(x) g(x)] \eta_i'(v) \dd v \dd x - \int A(u(x)) [\varphi_i(x) g(x)] \eta_i'(L_i) \dd v \right), \\
&= \sum_i \iint f(x,v) a(v) \cdot \nabla [\varphi_i(x) g(x)] \eta_i'(v) \dd v \dd x = \sum_i \int_{B_R} q_i(u(x)) \cdot \nabla [\varphi_i(x) g(x)] \dd x \geq 0.
\end{align*}
We used that $\eta_i'(L_i)=0$. The last inequality follows by the definition of entropy subsolution. Note that $\eta_i$ is non-decreasing only in $[L_i,+\infty)$, but $u$ does not take values below $L_i$ in the support of $\varphi_i$, and so the last inequality holds.

With this construction, we obtain two measures $\mu_0$ and $\mu_1$ so that $a(v) \cdot \nabla f = \partial_v \mu_0 - \mu_1$ and they are supported in $\bigcup_i G_i \times [L_i,U_i]$. In order to obtain a pair of measures $\mu_0$ and $\mu_1$ supported in $\{ (x,v) : \underline u(x) \leq v \leq \overline u(x)\}$, we use a sequence of open covers with balls of radius $r$ and pass to the weak-$\ast$ limit as $r \to 0$.
\end{proof}

The following two lemmas provide the most basic estimates on the total measure of $\mu_0$ and $\mu_1$. The first one (Lemma \ref{l:total-measure1}) is a simple consequence of integrating the equation with a suitable test function. The second one (Lemma \ref{l:total-measure0}) is analogous to an estimate in \cite{lions1994kinetic}.

\begin{lemma} \label{l:total-measure1}
Let $u \geq 0$ be an entropy subsolution to the conservation law \eqref{e:conservatiolaw-no-t} in $B_{1+\delta}$. Let $\mu_1$ be the measure as in Proposition \ref{p:kinetic-subsolutions}. Then
\[ \mu_1(B_{1} \times \R) \leq \left( \max |a| \right)  \delta^{-1} \|u\|_{L^1(B_{1+\delta})} .\]
\end{lemma}

\begin{proof}
We integrate the equation in Proposition \ref{p:kinetic-subsolutions} against the following test function, which is independent of $v$,
\[ g(x) = \begin{cases}
1 &\text{if } x \in B_1,\\
1 - \delta^{-1} (|x|-1) &\text{if } x \in B_{1+\delta} \setminus B_1, \\
0 &\text{if } |x| > 1+\delta.
\end{cases}\]
Thus,
\begin{align*} 
 \mu_1(B_1 \times \R) &\leq \int g \dd \mu_1 , \\ 
 &= \int_{B_{1+\delta} \times \R} f(x,v) a(v) \cdot \nabla g(x) \dd v \dd x, \\
&= \int_{B_{1+\delta}} [A(u(x)) - A(0)] \cdot \nabla g(x) \dd x, \\
&= \int_{B_{1+\delta}} (\max |a|) u(x) |\nabla g(x)| \dd x, \\
&\leq \frac{(\max |a|)}\delta \|u\|_{L^1(B_{1+\delta} \setminus B_1)}.
\end{align*}
\end{proof}

\begin{lemma} \label{l:total-measure0}
Let $u \geq 0$ be an entropy sub-solution to the conservation law \eqref{e:conservatiolaw-no-t} in $B_{1+\delta}$. The measure $\mu_0$ in Proposition \ref{p:kinetic-subsolutions} can be constructed so that there is a function $m \in L^\infty(\R) \cap L^1(\R)$ such that for any continuous function $\psi:\R \to \R$
\begin{equation} \label{e:tm0}
 \int_{B_{1} \times \R} \psi(v) \dd \mu_0 = \int_{\R} \psi(v) m(v) \dd v,
\end{equation}
and the function $m$ satisfies
\[ \|m\|_{L^\infty} \leq \left( \max |a| \right)  \delta^{-1} \|u\|_{L^1(B_{1+\delta})} .\]
Moreover, $m$ is supported in the interval $[0, \|u\|_{L^\infty}]$. 

Equivalently, for any $v \in [0,\infty)$ and $r>0$,
\[ \mu_0(B_1 \times [v,v+r]) \leq ( \max |a| )  \delta^{-1} r \|u\|_{L^1(B_{1+\delta})}.\]
\end{lemma}

The measures $\mu_0$ and $\mu_1$ in Proposition \ref{p:kinetic-subsolutions} are not necessarily unique. In Lemma \ref{l:total-measure0} we obtain an estimate which works for the measure $\mu_0$ which was constructed in the proof of Proposition \ref{p:kinetic-subsolutions}. It is not ruled out that there may be other two measures $\tilde \mu_0$ and $\tilde \mu_1$ so that $\partial_v \mu_0 - \mu_1 = \partial_v \tilde \mu_0 - \tilde \mu_1$ and $\tilde \mu_0$ does not satisfy the estimate \eqref{e:tm0}.

\begin{proof}
We use the same construction for $\mu_0$ as in the proof of Proposition \ref{p:kinetic-subsolutions} and verify that it satisfies \eqref{e:tm0}.

By splitting $\psi$ into its positive and negative parts, it is enough to prove the result assuming $\psi \geq 0$. 

Let $g : B_{1+\delta} \to \R$ be the same test function as in the proof of Lemma \ref{l:total-measure1}. The result of this lemma is a simple consequence of the following inequality
\begin{equation} \label{e:tm01}
 \langle \mu_0, g(x) \psi(v) \rangle \leq (\max |a|) \delta^{-1} \|u\|_{L^1} \|\psi\|_{L^1}.
\end{equation}
Thus, we concentrate the rest of the proof in verifying \eqref{e:tm01}

We consider an open covering of $B_{1+\delta}$ and a partition of unity $\{\varphi_i\}$ as in the proof of Proposition \ref{p:kinetic-subsolutions}. Following the same construction of $\mu_0$ relative to this open cover, we obtain
\[ \langle \mu_0, g(x) \psi(v) \rangle = \sum_i \int_{B_{1+\delta}} q_i(u(x)) \cdot [\varphi_i(x) g(x) ] \dd x.\]
Recall that $\eta_i$ is the function such that $\eta_i(L_i) = \eta_i'(L_i)=0$ and $\eta_i'' = \psi$. Moreover, $q_i$ is the function such that $q_i(L_i) = 0$ and $q'(v) = \eta_i'(v) a(v)$. Thus,
\begin{align*} 
 \langle \mu_0, g(x) \psi(v) \rangle &= \sum_i \int_{B_{1+\delta}} \left( \int_{L_i}^{u(x)} a(v) \int_{L_i}^v \psi(w) \dd w \dd v \right) \cdot\nabla [\varphi_i(x) g(x) ] \dd x, \\
\intertext{Exchanging the order of integration,}
&= \sum_i  \int_{L_i}^{\infty} \psi(w) \left( \int_{B_{1+\delta}} \one_{u(x) > w} [A(u(x)) - A(w)]  \cdot  \nabla [\varphi_i(x) g(x) ] \dd x \right) \dd w,\\
\intertext{Since $u$ is a subsolution, the innermost integral is positive for any $w \in \R$, then we get an inequality,}
&\leq \sum_i  \int_{0}^{\infty} \psi(w) \left( \int_{B_{1+\delta}} \one_{u(x) > w} [A(u(x)) - A(w)]  \cdot  \nabla [\varphi_i(x) g(x) ] \dd x \right) \dd w,\\
&=  \int_{0}^{\infty} \psi(w) \left( \sum_i \int_{B_{1+\delta}} \one_{u(x) > w} [A(u(x)) - A(w)]  \cdot  \nabla [\varphi_i(x) g(x) ] \dd x \right) \dd w,\\
&=  \int_{0}^{\infty} \psi(w) \left( \int_{B_{1+\delta}} \one_{u(x) > w} [A(u(x)) - A(w)]  \cdot  \nabla g(x)  \dd x \right) \dd w,\\
\end{align*}

Setting 
\[ m(w) = \int_{B_{1+\delta}} \one_{u(x) > w} [A(u(x)) - A(w) \cdot  \nabla g(x)  \dd x,\]
we clearly have $\|m\|_{L^\infty} \leq (\max |a|) \delta^{-1} \|u\|_{L^1} $.

This estimate holds for any open cover $\{G_i\}$ and therefore it will hold after passing to the limit.
\end{proof}

\section{Averaging estimates for subsolutions}
\label{s:averaging}

The following is a relatively standard averaging Lemma. We can find its proof for example in Theorem A in \cite{lions1994kinetic} or Lemma 2.1 in \cite{tadmor2007velocity}.

\begin{prop} \label{p:averaging}
Assume that the nonlinearity satisfies Assumption \ref{a:genuine-nonlinearity}. There exists a constant $\theta>0$ depending on $\alpha$, such that for any $s \in (0,\theta)$ and $1/r = (1+\theta)/2$, the following estimate holds. 

Given any nonnegative subsolution $u$ of \eqref{e:conservatiolaw-no-t} in certain domain $\Omega$, let $f$, $\mu_0$ and $\mu_1$ be as in Proposition \ref{p:kinetic-subsolutions}. Let $\varphi: \R^d \to \R$ be smooth and supported inside a ball of radius less or equal to $2$ contained in $\Omega$. Then
\[ \|u \varphi\|_{W^{s,r}} \leq C \|f(x,v) \varphi(x)\|_{L^2}^{1-\theta} \left( \int \varphi(x) \dd \mu_0 + \int \varphi(x) \dd \mu_1 + \int |a(v) \cdot \nabla \varphi(x)| f(x,v) \dd x \dd v \right)^\theta,\]
where $C$ is a constant depending only on  $\|u\|_{L^\infty}$, $s$, the constants $\alpha$ and $C$ in Assumption \ref{a:genuine-nonlinearity} and $\|a\|_{C^1(I)}$
\end{prop}

Applying this averaging estimate in the case of conservation laws, we deduce some fractional Sobolev regularity and improved order of integrability in the following corollary. The estimate in the following corollary is by no means optimal. However, for the purpose of the proofs in this paper, the optimal regularity exponents are not necessary.

\begin{cor} \label{c:averaging}
Assume that the nonlinearity satisfies Assumption \ref{a:genuine-nonlinearity}.  Let $u$ be a non-negative entropy subsolution of \eqref{e:conservatiolaw-no-t} in $B_R$ with $R \leq 2$. Let $r<R$. Then
\[ \|u\|_{L^p(B_r)} \leq C \|u\|_{L^1(B_R)}^{(1+\theta)/2} (R-r)^{-\theta},\]
for any $p \geq 1$ so that $1/p > (1+\theta)/2 - \theta/d$. Like in Proposition \ref{p:averaging}, $\theta>0$ depends on $\alpha$ and the constant $C$ depends on $\|u\|_{L^\infty}$, $p$ and the parameters in Assumption \ref{a:genuine-nonlinearity}.
\end{cor}

\begin{proof}
Let $\varphi: \R^d \to \R$ be smooth, supported in $B_R$, equal to one in $B_r$, with $|\nabla \varphi| \leq (R-r)^{-1}$. Note that since $0 \leq f(x,v) \varphi(x) \leq 1$ for every $(x,v) \in \Omega \times \R$, we have
\[ \| f(x,v) \varphi(x) \|_{L^2} \leq \| f(x,v) \varphi(x) \|_{L^1}^{1/2} = \|u \varphi \|_{L^1}^{1/2}.\]
We apply Proposition \ref{p:averaging} with this $\varphi$, using the estimates on $\mu_0$ and $\mu_1$ given in Lemmas \ref{l:total-measure0} and \ref{l:total-measure1}. We obtain
\[ \|\varphi u\|_{W^{s,r}(\R^d)} \leq C \|u(x)\|_{L^1(B_R)}^{(1+\theta)/2} (R-r)^{-\theta},\]
for some constant $C$ depending on $\|u\|_{L^\infty}$. Then, we apply the Sobolev inequality $W^{s,r} \subset L^p$.
\end{proof}

We now give another corollary which follows immediately from the previous one using the H\"older's inequality: $\|u\|_{L^1(B_r)} \leq \|u\|_{L^p(B_r)} |\{u>0\} \cap B_r|^{1/p'}$.

\begin{cor} \label{c:averaging2}
Assume that the nonlinearity satisfies Assumption \ref{a:genuine-nonlinearity}. Let $u$ be a non-negative entropy subsolution of \eqref{e:conservatiolaw-no-t} in $B_R$ with $R \leq 2$. Let $r<R$. Then the following inequality holds
\[ \|u\|_{L^1(B_r)} \leq C (R-r)^{-\theta} \|u\|_{L^1(B_R)}^{(1+\theta)/2} \left\vert \{u>0\} \cap B_r \right\vert^{1/p'}.\]
Here $p'$ is any positive number so that $1/p' < (1-\theta)/2 + \theta/d$ and $\theta>0$ depends on $\alpha$.  The constant $C$ depends on $\|u\|_{L^\infty}$, $p'$ and the parameters in Assumption \ref{a:genuine-nonlinearity}.
\end{cor}

The precise value of $\theta$ in Corollary \ref{c:averaging2} is not important. The key point in order to make De Giorgi iteration work out is that the exponents $(1+\theta)/2$ and $1/p'$ add up to a quantity larger than one.

The special upper bound ``$R \leq 2$'' is not essential. However, considering larger supports for the function would affect the constants in the right hand side of the averaging estimates.


\section{De Giogi iteration}
\label{s:deGiorgi}

In this section we obtain a bound on the oscillation of the solution $u$ in some ball $B_1$, in terms of its $L^1$ norm in a larger ball $B_2$. This estimate is the key for the proofs of all the other results in this paper. The proof is based on an iteration procedure inspired by the method of De Giorgi \cite{e1957sulla}.

An immediate consequence of the following Theorem is that if a uniformly bounded sequence of solutions to \eqref{e:conservatiolaw-no-t} converges to a constant in $L^1_{loc}$, then it also converges in $L^\infty_{loc}$.

\begin{thm} \label{t:degiorgi}
Assume that the nonlinearity satisfies the assumption \ref{a:genuine-nonlinearity} and let $M$ be a nonnegative number.  There are constants $\gamma>0$ and $C$ so that for any non-negative entropy subsolution $u$ of \eqref{e:conservatiolaw-no-t} in $B_2$, with $\|u\|_{L^\infty(B_2)} \leq M$, the following estimate holds
\[ \|u\|_{L^\infty(B_1)} \leq C \|u\|_{L^1(B_2)}^{\gamma}.\]
The constant $\gamma$ depends on $\alpha$ and $d$, $C$ depends on the constants involved in Assumption \ref{a:genuine-nonlinearity} and $M$.
\end{thm}

\begin{proof}
We want to prove that $\|u\|_{L^\infty(B_1)} \leq U$ for a suitable value of $U$. Let 
\begin{align*}
\ell_k &:= (1-2^{-k}) U, \\
u_k &:= (u-\ell_k)_+, \\
r_k &:= 1+2^{-k}, \\
A_k &:= \int_{B_{r_k}} u_k \dd x
\end{align*}
The objective of the proof is find a value of $U$ such that $A_k \to 0$ as $k \to +\infty$, implying the conclusion of the Theorem. The numbers $A_k$ form a decreasing sequence of nonnegative numbers.

Because of Lemma \ref{l:max-of-subsolutions}, each function $u_k$ is a subsolution of
\[ a(u_k(x) + \ell_k) \cdot \nabla u_k(x) \leq 0.\]
Since the image of $u_k + \ell_k$ is contained in the image of $u$, these equations satisfy Assumption \ref{a:genuine-nonlinearity} uniformly in $k$.

Applying Corollary \ref{c:averaging2} to $u_{k+1}$, with $r=r_k$ and $R=r_{k+1}$, we get
\[ \|u_{k+1}\|_{L^1(B_{r_{k+1}}) } \leq C 2^{\theta k}  \| u_{k+1} \|_{L^1(B_{r_k})}^{(1+\theta)/2} |\{ u_{k+1} > 0\} \cap B_{r_{k+1}} |^{1/p'}.\]

Note that by construction, $u_{k+1} \leq u_k$ and $u_{k+1} > 0$ only where $u_k > 2^{-k-1} U$. Therefore, applying Chebyshev's inequality,
\[ \|u_{k+1}\|_{L^1(B_{r_{k+1}}) } \leq C 2^{\theta k + k/p'} \| u_{k} \|_{L^1(B_{r_k})}^{(1+\theta)/2 + 1/p'} U^{-1/p'}.\]

Observe that choosing $1/p'$ sufficiently close to $(1-\theta)/2 + \theta/d$, the exponent $1+\delta := (1+\theta)/2 + 1/p'$ will be larger than one. We are left with,
\[ \|u_{k+1}\|_{L^1(B_{r_{k+1}}) } \leq C 2^{Ck} \| u_{k} \|_{L^1(B_{r_k})}^{1+\delta} U^{-1/p'}.\]

Then, 
\[ \frac{A_{k+1}}{U^{1/(\delta p')}} \leq C 2^{Ck} \left( \frac{A_{k+1}}{U^{1/(\delta p')}} \right)^{1+\delta}.\]

Thus, the nonincreasing sequence $A_k$ converges to zero provided that $A_0 / U^{1/(\delta p')} \leq c_0$ for some sufficiently small constant $c_0$. Therefore, choosing $U = (c_0^{-1} A_0)^{\delta p'}$, we get $A_k \to 0$ as $k \to \infty$ and we conclude that
\[ \|u\|_{L^\infty(B_1)} \leq U = C \|u\|_{L^1(B_2)}^{\delta p'}.\]
\end{proof}

\begin{remark} \label{r:causality}
For $t$-dependent equations as in \eqref{e-conservationlaw}, there is a predetermined order of causality and the estimate in Theorem \ref{t:degiorgi} takes the slightly more precise form
\[ \|u\|_{L^\infty([1,2] \times B_1)} \leq C \|u\|_{L^1([0,2] \times B_2)}^\gamma,\]
for any solution $u$ of \eqref{e-conservationlaw} in $[0,2] \times B_2$.
\end{remark}

\begin{remark} It may seem strange at first that the constant $C$ in the right hand side of Theorem \ref{t:degiorgi} depends on the upper bound $M$ for $\|u\|_{L^\infty(B_2)}$. Trivially, $\|u\|_{L^\infty(B_1)} \leq \|u\|_{L^\infty(B_2)} \leq M$. In our applications of this lemma, we will set $M$ to be the global $L^\infty$ norm of $u$. We work with bounded entropy solutions. Theorem \ref{t:degiorgi} is useful when $\|u\|_{L^1(B_2)}$ is very small, making the upper bound for $\|u\|_{L^\infty(B_1)}$ smaller than $M$. For example, Theorem \ref{t:degiorgi} implies that if we have an equibounded sequence of entropy solutions of \eqref{e:conservatiolaw-no-t} that converges to a constant in $L^1_{loc}$, then it also converges locally uniformly.

Most estimates in this paper depend on the $L^\infty$ norm of the solution. To start with, the genuine nonlinearity condition given in Assumption \ref{a:genuine-nonlinearity} can only hold (with bounded values of the constant $C$ and $\|a\|_{C^1}$) when $u$ takes values in a bounded interval.
\end{remark}

\section{Continuity outside of the jump set}
\label{s:continuity}

The Haussdorf dimension of the jump set $J$ (defined in \eqref{e:shock-set}) is at most $d-1$. It is proved in \cite{de2003structure} (Theorem 1) that $J$ is codimension-one rectifiable and that for every point $x \notin J$, the function $u$ is $VMO$ at the point $x$ in the sense that
\begin{equation} \label{e:VMO-at-x}
 \lim_{r \to 0} \frac 1 {r^d} \int_{B_r(x)} |u(y) - u_r(x)| \dd y = 0.
\end{equation}
where
\[ u_r(x) = \frac 1 {|B_r|} \int_{B_r(x)} u(y) \dd y.\]

The results in \cite{de2003structure} are based on the analysis of blow-up limits of entropy solutions of \eqref{e:conservatiolaw-no-t}. A blow up sequence is a sequence of rescalings of the solution that \emph{zoom in} near a point. From classical averaging techniques, we know we can always extract a subsequence that converges in $L^1_{loc}$. The set $J$ is defined so that the entropy dissipation measure of any blow-up limit centered outside of $J$ will vanish. Therefore, a Liouville theorem proved in \cite{de2003structure} allows them to show that any blow-up limit at a point $x \notin J$ has to be constant. It is easy to see that this is equivalent to \eqref{e:VMO-at-x}. In this context, establishing the uniqueness of these blowup limits would correspond to $x$ being a Lebesgue point of $u$. Moreover, proving that the blow-up limit converges in $L^\infty_{loc}$ instead of $L^1_{loc}$ would be equivalent to the continuity of $u$ at $x$.

In this section we will combine the information in \eqref{e:VMO-at-x} with Theorem \ref{t:degiorgi} to prove that $u$ is in fact continuous at every point $x \notin J$. Theorem \ref{t:intro-continuity} follows directly from the following proposition.

\begin{prop} \label{p:continuous-out-of-J}
Let $u$ be an entropy solution to the equation \eqref{e:conservatiolaw-no-t} in some domain $\Omega$ and let Assumption \ref{a:genuine-nonlinearity} hold. Let $x \in \Omega \setminus J$. Then the function $u$ is continuous at the point $x$ in the sense that
\[ \lim_{r \to 0} \left( \esssup_{B_r(x)} u - \essinf_{B_r(x)} u \right) = 0.\]
In other words, $\overline u(x) = \underline u(x)$ for every $x \notin J$.
\end{prop}

\begin{proof}
Let $u_r(x)$ be as in \eqref{e:VMO-at-x} holds. For each fixed value of $x \in \Omega \setminus J$ and $r>0$, let us define $u_1(y) = (u(y)-u_r(x))_+$ and $u_2(y) = (u_r(x) - u(y))_-$. Clearly, we have 
\[ \osc_{B_r(x)} u \leq \esssup_{B_r(x)} u_1 + \esssup_{B_r(x)} u_2.\]

By Lemma \ref{l:max-of-subsolutions}, the function $u_1$ is an entropy subsolution of the equation
\[ a(u_1(y) + u_r(x)) \cdot \nabla_y u_1(y) \leq 0.\]
Moreover, $u_2$ is an entropy subsolution of 
\[ a(u_r(x) - u_2(y)) \cdot \nabla_y u_1(y) \leq 0.\]

 Applying Theorem \ref{t:degiorgi} and scaling, for $i=1,2$, we get
\[
 \esssup_{B_{r/2}(x)} u_i \leq \left(  \frac C {r^d} \int_{B_r(x)} u_i(x) \dd x \right)^\gamma
\leq \left( \frac C {r^d} \int_{B_r(x)} |u(x) - a(r)| \dd x\right)^\gamma \to 0,
\]
which concludes the proof.
\end{proof}

\begin{remark}
As we mentioned in the introduction, the result of Theorem \ref{t:intro-continuity} is new even for $BV$ solutions. However, when the initial data $u_0$ is sufficiently smooth and decaying at infinity, there is an easier proof that solutions of the conservation law equation \eqref{e-conservationlaw} are continuous almost everywhere. We sketch this proof in this remark. We define the sup and inf convolutions of the entropy solution $u$ by the following formulas.
\begin{align}
u^\eps(t,x) = \esssup_{|y-x| < \eps} u(t,y), \\
u_\eps(t,x) = \essinf_{|y-x| < \eps} u(t,y).
\end{align}
Since we assume that the initial data $u(0,x) = u_0(x)$ is smooth and has sufficient decay at infinity, we have
\[ \int_{\R^d} u^\eps(0,x) - u_\eps(0,x) \dd x \lesssim \eps,\]
for every $\eps >0$.

Using Lemma \ref{l:max-of-subsolutions}, it is possible to prove that $u^\eps$ is an entropy subsolution and $u_\eps$ is an entropy supersolution to the equation \eqref{e-conservationlaw}. Therefore, 
\[ \int_{\R^d} u^\eps(t,x) - u_\eps(t,x) \dd x \lesssim \eps,\]
for every $\eps >0$ and $t>0$. Taking $\eps \to 0$ and applying the Monotone convergence theorem, we see that $\overline u(t,x) = \underline u(t,x)$ almost everywhere. Thus, $u$ is continuous almost everywhere.

In the proof we described above, we are taking $\overline u(t,\cdot)$ and $\underline u(t,\cdot)$ to be the upper and lower semicontinuous envelopes of $u(t,\cdot)$ for every fixed $t$. It is possible to prove that they coincide with the corresponding envelopes in space-time using the finite speed of propagation for conservation laws.
\end{remark}

\section{Large time decay in $L^\infty$}
\label{s:decay}

The following result is a direct application of Theorem \ref{t:degiorgi} after scaling.

\begin{lemma} \label{l:decay}
Let $u_0 : \R^d \to \R$ be the initial condition for the equation \eqref{e-conservationlaw} and let Assumption \ref{a:genuine-nonlinearity} hold. Let us assume that $u_0 \in L^\infty \cap L^1(\R^d)$. Then, if $u$ solves \eqref{e-conservationlaw}, we have
\[ |u(t,x)| \leq C t^{-d \gamma} \|u_0\|_{L^1}^{\gamma}.\]
The constants $C$ and $\gamma$ are the same as in Theorem \ref{t:degiorgi} (applied in $\R^{d+1}$).
\end{lemma}

\begin{proof}
Note that from the usual maximum principle and $L^1$ contraction principle, we have that for all $t > 0$, $\|u(t,\cdot)\|_{L^1(\R^d)} \leq \|u_0\|_{L^1(\R^d)}$ and $\|u\|_{L^\infty((0,\infty) \times \R^d)} \leq \|u_0\|_{L^\infty(\R^d)}$.

For any fixed $t >0$ and $x \in \R^d$, we define the function $w(y)$, for $y \in B_2 \subset \R^{d+1}$, by the formula,
\[ w(y) = \max(u(t + ty_0/2, x_1+ ty_1/2, \cdots, x_d + t y_d/2),0). \]
We apply Theorem \ref{t:degiorgi} to $w$. By Lemma \ref{l:max-of-subsolutions}, the function $w$ is an entropy subsolution of
\[ (1,a(w))  \cdot \nabla w \leq 0 \qquad \text{ in } B_2 \subset \R^{d+1}.\]
Note that here $B_2$ is the ball in $\R^{d+1}$. According to Theorem \ref{t:degiorgi},
\[ u(t,x) \leq w(0) \leq C \|w\|_{L^1(B_2)}^{\gamma} \leq C \left( \frac 1 {t^{d+1}} \int_0^{2t} \int u_+(s,x) \dd x \dd s \right)^{\gamma} \leq C t^{-d\gamma} \|u_0\|_{L^1}^{\gamma}.\]
Similar reasoning shows that $u(t,x) \geq -C t^{-d\gamma} \|u_0\|_{L^1}^{\gamma}$.
\end{proof}

Now we refine the decay in $\|u(t,\cdot)\|_{L^\infty}$ using the two-parameter scaling defined in section \ref{s:scaling}.

Theorem \ref{t:intro-decay} is simply a restating of the following theorem. Note that here $(t,x) \in \R^{d+1}$ and for any value of $m$, $\det S_\lambda \approx \lambda^q$ where $q = j_1 + \dots + j_{d+1}$ is the sum of the indexes in the construction of $S_\lambda$. Since for $t$-dependent equations, $a(0) = (1,\dots)$, we will have $j_1 = 0$. The other indexes could be any number in that range, therefore $q \leq m (m-1) \dots (m-d+1) = d (2m-d+1)/2$.

\begin{thm} \label{t:decay}
Let Assumption \ref{a:hormander} hold. Let $\gamma_0$ be the positive number such that
\[ \det S_\lambda = C \lambda^{1/\gamma_0-1}.\]

Let us assume that $u_0 \in L^\infty \cap L^1(\R^d)$. Then, if $u: [0,\infty) \times \R^d \to \R$ is the solution to the equation \eqref{e-conservationlaw} with initial value $u_0$, we have, for any $\gamma \in (0,\gamma_0)$,
\begin{equation} \label{e:decay}
 |u(t,x)| \leq C t^{-d \gamma} \|u_0\|_{L^1}^{\gamma},
\end{equation}
for a constant $C$ depending on $\|u_0\|_{L^\infty}$, $\gamma$ and the constants in Assumption \ref{a:genuine-nonlinearity}. 

In particular, for the generalized Burgers equation we have
\[ \gamma_0 = \left( 1 + \frac{d (d+1)}2 \right)^{-1},\]
and for any equation satisfying Assumption \ref{a:hormander},
\[ \gamma_0 \geq \left( 1 + \frac{   d (2m-d+1)  }2 \right)^{-1},\]
\end{thm}

\begin{proof}
We start with the estimate from Lemma \ref{l:decay} and improve the exponent $\gamma$ iteratively using scaling.

Indeed, assume that we know that \eqref{e:decay} holds for some $\gamma>0$. Evaluating this at $t/2$, we get
\[ \|u(t/2,\cdot) \|_{L^\infty} \leq C t^{-d \gamma} \|u_0\|_{L^1}^{\gamma} \]
Set $\lambda = C t^{-d \gamma} \|u_0\|_{L^1}^{\gamma}$. We apply the scaling transformation $S_\lambda$ described in \eqref{e:two-parameter-scaling}. Note that for $t$-dependent equations, $S_\lambda$ does not affect the variable $t$. We abuse notation by writing $S_\lambda$ as a linear transformation from $\R^d \to \R^d$ (as opposed to $\R^{d+1} \to \R^{d+1}$ fixing the first component). The function $\tilde u$ given by
\[ \tilde u(s,x) = \lambda^{-1} u(t/2 + s, S_\lambda x)\]
satisfies a conservation law equation satisfying Assumption \ref{a:genuine-nonlinearity} with uniform constant provided that $\lambda < v_0$ (defined in section \ref{s:scaling}). Moreover, $\|\tilde u\|_{L^\infty([0,\infty) \times \R^d)} \leq 1$.

If $\lambda \geq v_0$, then that means that $t^{-d} \|u_0\|_{L^1} \geq c_0$ for some constant $c_0$, and there is nothing to prove. So, we can safely assume $\lambda < v_0$.

We apply Lemma \ref{l:decay} again to the function $\tilde u$ after time $t/2$. We obtain
\[ |\tilde u(t/2,x) | \leq C t^{-d \gamma} \|\tilde u(0,\dots)\|_{L^1}^{\gamma} = C t^{-d\gamma} \left( \lambda^{-1} \det S_\lambda ^{-1} \|u_0\|_{L^1}\right)^\gamma = C \left( t^{-d} \|u_0\|_{L^1} \right)^{\gamma - \gamma^2/\gamma_0 }.\]
Then,
\[ |u(t,x) | \leq \lambda |\tilde u(t/2,x)| = C \left( t^{-d} \|u_0\|_{L^1} \right)^{2\gamma - \gamma^2/\gamma_0 }.\]

Therefore, we improve the exponent from $\gamma$ to $2\gamma - \gamma^2 / \gamma_0$. Iterating this procedure, $\gamma$ will approximate arbitrarily the fixed point of the map, $\gamma_0$.
\end{proof}

\begin{remark} \label{r:optimal-decay}
The exponent $\gamma_0$ is optimal at least when $d=1$. For the one dimensional problem $u_t + u^m u_x = 0$ (which corresponds to Burgers equation when $m=1$), we obtain $\gamma_0=(m+1)^{-1}$. Indeed, it is achieved by the example
\[ u(t,x) = \begin{cases}
\left( \frac x {t+1} \right)^{1/m} & \text{ for } x \in [0,(t+1)^{-m/(m+1)}], \\
0 & \text{otherwise}.
\end{cases}\]
In this case $\max_x u(t,x) = (t+1)^{-1/(m+1)}$ for any $t>0$.

When $d>1$, the optimality of the exponent $\gamma_0$ is currently unclear.
\end{remark}

\begin{remark} \label{r:decay-with-Linfty-norm}
In many cases, we can compute how the constant $C$ in Theorem \ref{t:intro-decay} depends on $\|u_0\|_{L^\infty}$. This is possible whenever the scaling $S_\lambda$ described in section \ref{s:scaling} can be applied in the full range of the parameter $\lambda \in (0,+\infty)$.

Let us consider for example the $d$-dimensional generalized Burgers equation,
\[ u_t + (u,u^2,\dots,u^d) \cdot \nabla u = 0.\]
In this case $S_\lambda(t,x_1,\dots,x_d) = (t,\lambda x_1, \lambda^2 x_2, \dots, \lambda^d x_d)$. The function $u_{r,\lambda}(t,x) = \lambda^{-1} u(S_\lambda(rt,rx))$ solves the same equation as $u$.

According to Theorem \ref{t:intro-decay}, we have that for some constant depending on $\|u_0\|_{L^\infty}$, 
\[ \|u(t,\cdot)\|_{L^\infty(\R^d)} \leq C \|u_0\|_{L^1}^\gamma t^{-d\gamma}.\]
We will pick $\lambda = \|u_0\|_{L^\infty}$ so that the transformed solution $u_{1,\lambda}$ has initial $L^\infty$ norm equal to one. Thus, we get
\begin{align*}
\|u_{1,\lambda}(0,\cdot)\|_{L^\infty(\R^d)} &= 1 ,\\
\|u_{1,\lambda}(0,\cdot)\|_{L^1(\R^d)} &= \lambda^{-1} \|u_0\|_{L^1} / \det(S_\lambda) = \|u_0\|_{L^\infty}^{-1-d(d+1)/2} \|u_0\|_{L^1}, \\
\|u(t,\cdot)\|_{L^\infty} &= \lambda \|u_{1,\lambda}(t,\cdot)\|_{L^\infty} \leq C_1 \|u_0\|_{L^\infty} \left(\|u_0\|_{L^\infty}^{-1-d(d+1)/2} \|u_0\|_{L^1} \right)^\gamma t^{-d \gamma}.
\end{align*}

Thus, we deduce that
\[ |u(t,x)| \leq C_1 \| u_0 \|_{L^\infty}^{1 - \gamma \left( 1 + d (d+1) /2 \right)} \|u_0\|_{L^1}^\gamma t^{-d \gamma},\]
where $C_1$ depends on dimension only.

Note that, according to Theorem \ref{t:intro-decay}, $\gamma < \gamma_0 = \left( 1 + d (d+1) /2 \right)^{-1}$.
\end{remark}

\section{Characteristics curves}
\label{s:characteristics}

In this section we will construct certain Lipschitz curves along which an entropy solution $u$ attains certain constant value. It is convenient to consider $t$-dependent equations as in \eqref{e-conservationlaw}, since the parameter $t$ will also serve as the parametrization of these curves.

There are other similar, but not equivalent, definitions of generalized characteristics. See for example \cite{dafermos1977generalized}.



We will use a precise version of the local maximum principle which we give in the following lemma.


\begin{lemma} \label{l:local-max-pple}
Let $u:[0,T] \times \Omega \to \R$ be an entropy subsolution of \eqref{e-conservationlaw}. Assume that $u$ takes values in some interval $I$ and $K$ is the convex envelope of $\overline{\{a(v) : v \in I\}}$. Then, for $t> \tau >0$ and $x\in\Omega$,
\[ \overline u(t,x) \leq \max_{x - \tau K} \overline u(t-\tau,\cdot),\]
provided that $x - \tau K \subset \Omega$. Also, if $u$ is a supersolution of \eqref{e-conservationlaw}, 
\[ \underline u(t,x) \geq \min_{x - \tau K} \underline u(t-\tau,\cdot).\]
In particular, if $u$ is an entropy solution, then for every value $v_0 \in [\underline u(t,x), \overline u(t,x)]$, there exists a $y \in x - \tau K$ such that $v_0 \in [\underline u(t-\tau,y), \overline u(t-\tau,y)]$
\end{lemma}

\begin{proof}
Let us fix a point $(t_0,x_0)$ and $\tau \in (0,t_0)$. Let \[M := \max_{x_0 - \tau K} \overline u(t_0-\tau,\cdot) = \lim_{\delta \to 0} \esssup \{ u(t_0-\tau,y) : \dist(y,x_0-\tau K) < \delta\}.\]

Let $\eps>0$ be arbitrary, and let us pick $\delta>0$ so that $u(t_0-\tau,y) \leq M+\eps$ whenever $\dist(y,x_0-\tau K) < 2\delta$.

Let $\varphi_\delta$ be the following function
\[ \varphi_\delta(x) = \begin{cases}
1 & \text{if } x \in K, \\
0 & \text{if } \dist(x,K) > \delta, \\
1 - \delta^{-1} \dist(x,K) &\text{otherwise.}
\end{cases}
\]
Consider the quantity
\[ I(t) := \int_{\R^d} \varphi_\delta \left( \frac {x_0-x} {t_0+\delta-t} \right) (u(t,x) - M - \eps)_+ \dd x.\]
From the definition of $M$, $\eps$ and $\delta$, we know that $I(t_0 - \tau) = 0$.

Since $u$ is an entropy subsolution, for $t \in [t_0-\tau, t_0+\delta]$,
\begin{align*} 
 I'(t) &\leq \int_{\R^d} \frac{x_0-x}{(t_0+\delta-t)^2} \cdot \nabla \varphi_\delta \left( \frac{x_0- x}{(t_0+\delta-t)} \right) (u(t,x) - M-\eps)_+ \\
& \phantom{=} - \frac{1}{(t_0+\delta-t)}  \nabla \varphi_\delta \left( \frac{x_0 - x}{(t_0+\delta-t)} \right)  \cdot (A(u(t,x)) - A(M+\eps)) \one_{u>M+\eps} \dd x.
\end{align*}

Therefore
\[ I'(t) \leq \frac 1 {t_0+\delta-t} \int_{\R^d} \max_{V \in K} \left\{ \left( \frac{x_0-x}{t_0+\delta-t} - V \right) \cdot \nabla \varphi_\delta\left( \frac{x_0-x}{t_0+\delta-t} \right)  \right\} (u-M -\eps)_+ \dd x.\]

Let $y := (x_0-x) / (t_0 +\delta -t)$. We observe that $\nabla \varphi(y) = \delta^{-1} (z-y)/|z-y|$ whenever $0 < \dist(y,K) < \delta$ and $z \in K$ is the point in $K$ such that $\dist(y,K) = |y-z|$. Moreover, the maximum in the integrand is achieved at $V = z$. Thus,
\[ \max_{V \in K} \left\{ \left( \frac{x}{t_0+\delta-t} - V \right) \cdot \nabla \varphi_\delta\left( \frac{x}{t_0+\delta-t} \right)  \right\} = (y-z) \cdot \nabla \varphi_\delta(y) = -|y-z|.\]
Therefore, $I'(t) \leq 0$. Since $I(t)$ is always nonnegative, we have $I(t) = 0$ for all $t \in [t_0 -\tau,t_0 +\delta]$, which means that $u(t,x) \leq M + \eps$ almost everywhere in a neighborhood of $x_0 - (t_0 + \delta - t) K$ for every $t \in [t_0-\tau,t_0 + \delta]$.

 This implies that $\overline u(t,x) \leq M + \eps$ whenever $t \in [0,t_0+\delta)$ and $x$ belongs to  $x_0 - (t_0+\delta-t) K$. In particular $\overline u(t_0,x_0) \leq M$. Since $\eps>0$ is arbitrarily small, we get that $\underline u(t_0,x_0) \leq M$.

The other inequality follows similarly.

Finally, notice that the sets $\{ y \in x_0 - \tau K : \underline u(t_0-\tau,y) \leq v_0\}$ and $\{ y \in x_0 - \tau K : \overline u(t_0-\tau,y) \geq v_0\}$ are closed, nonempty, and their union is the full convex set $x_0 - \tau K$. By connectedness, there must be some point in their intersection.
\end{proof} 


\begin{prop} \label{p:characteristics}
Let $u : [0,T] \times \Omega \to \R$ be an entropy solution of \eqref{e-conservationlaw}. Let $\bar u$ and $\underline u$ be its upper and lower semicontinuous envelopes. For any point $x_0 \in \Omega$ and $v_0 \in [\underline u(T,x_0), \overline u(T,x_0)]$ there exists a Lipschitz curve $\gamma: [0,T] \to \R^n$ such that
\begin{align*}
\gamma(T) &= x_0, \\
\gamma'(t) &\in \hull{ a([\underline u(t,\gamma(t)), \bar u(t,\gamma(t))]) }, \text{wherever $\gamma'$ exists,}\\
v_0 &\in [\underline u(t,\gamma(t)), \overline u(t,\gamma(t))], \qquad \text{for all } t \in [0,T].
\end{align*}
\end{prop}

\begin{proof}
Without loss of generality, we do the proof for $T=1$.

Since $u$ is a bounded function, let us define $M = \sup |a(I)|$, where $I$ is a closed interval which contains the range of $u$. The Lipschitz curve $\gamma$ that we construct will have a Lipschitz constant less or equal to $M$.

We will construct a sequence of approximations of $\gamma$, which we call $\gamma_k$. Each $\gamma_k$ is a polygonal, with vertexes at the point $t = j T/k$ with $j = 0,1,\dots,k$.

We start by describing the construction of $\gamma_k$, for each $k= 1,2, \dots$.

We define $\gamma_k(1) = x_0$ for any value of $k$. 

We will determine the values of $\gamma_k(j/k)$ iteratively for $j = k, k-1, k-2, \dots, 0$ (in that order). 

For each fixed value of $k$, let us call $x_j = \gamma(j/k)$, $t_j = j/k$. 

Let us supposed that we have established a value of $x_j$ so that $|x_j - x_{j+1}| < M/k$ so that $v_0 \in [\underline u(t_j,x_j), \overline u(t_j,x_j)]$. We find $x_{j-1}$ applying Lemma \ref{l:local-max-pple} (note that this choice may not be unique). We have that $x_j - x_{j-1} \in \frac 1 k \hull{a(I)}$. In particular $|x_j-x_{j-1}| \leq M/k$, so the polygonal curve $\gamma_k$ is Lipschitz with constant $M$. Applying Lemma \ref{l:local-max-pple} again if necessary, we can ensure that 
\[ x_j - x_{j-1} \in \frac 1 k \hull{\overline{a(u(D_j))}},\] where $D_j := [t_{j-1} , t_j] \times B_{M/r}(x_j)$.

Let $d_j := k(x_j - x_{j-1}) \in \hull{\overline{a(u(D_j))}}$.

The polygonals $\gamma_k$ are uniformly Lipschitz. Therefore, there is a uniformly convergent subsequence. We abuse notation by still calling this subsequence $\gamma_k$. Let $\gamma$ be its uniform limit. Since $\gamma$ is Lipschitz, it is differentiable almost everywhere and for any $0 \leq \bar t_1 < \bar t_2 \leq T$,
\begin{align*} 
 \gamma(\bar t_1) - \gamma(\bar t_2) &= \lim_{k \to \infty} \gamma_k([kt_2]/k) - \gamma_k([kt_1]/k) , \\ 
 &= \lim_{k \to \infty} \sum_{j \in [k \bar t_1]}^{[k \bar t_2]} \frac 1 k d_j,\\
&= \lim_{k \to \infty} \frac{[k \bar t_2] - [k \bar t_1]}k V_k = (t_2 - t_1) V.
\end{align*}
Here, each vector $V_k$ is in the convex envelope of the union of the sets $a(\overline{u(D_j)})$ for $j=0,\dots,k$. Then $V_k \in \hull{ a([R_k, S_k]) }$, where
\begin{align*} 
 R_k &:= \essinf \{ u(t,y) : t \in [\bar t_1, \bar t_2], |y-\gamma_k(t)| < 2M/k\}, \\
 S_k &:= \esssup \{ u(t,y) : t \in [\bar t_1, \bar t_2], |y-\gamma_k(t)| < 2M/k\}.
\end{align*} 
After taking the limit $k \to \infty$, we get that $V = \lim_{k \to \infty} V_k$ satisfies $V \in \hull{ \overline { a([R_\infty, S_\infty]) }}$, where
\[ R_\infty := \min \{ \underline u(t,x)) : t \in [\bar t_1, \bar t_2], x = \gamma(t)\} \qquad \text{and} \qquad  S_\infty := \max \{  \overline u(t,x)) : t \in [\bar t_1, \bar t_2], x = \gamma(t) \}.\] 

Therefore, at any point $t$ where $\gamma$ is differentiable, we must have that $\gamma'(t)$ belongs to the convex envelope of $a([\underline u(t,\gamma(t)), \overline u(t,\gamma(t))])$.
\end{proof}

\section{Lack of entropy dissipation away from the jump set}
\label{s:no-entropy}

The result in this section says that when $u$ is continuous in an open domain $\Omega$, there is no entropy dissipation there. Combining Lemma \ref{l:continuous-solutions} with Theorem \ref{t:intro-continuity}, we deduce Theorem \ref{t:intro-no-entropy}.

In \cite{dafermos2006}, C. Dafermos proves for one dimensional problems that continuous solutions do not dissipate entropy. In this section we generalize that result to multidimensional conservation laws following a similar approach. The main idea is that for continuous solutions the characteristic curves are well defined. We can prove that they are straight lines, they extend backwards and forward throughout the domain of the equation, and they do not cross.

We start with a preparatory lemma about the characterisitic curves of a continuous solution.

\begin{lemma} \label{l:characterisitcs-two-sided}
Let $\Omega \subset \R^d$ be open and $u : \Omega \to \R$ be a continuous function which solves \eqref{e:conservatiolaw-no-t}. For any $x_0 \in \R^d$, the function $u$ is constant along the segment $\{x_0 + t a(u(x_0)): t \in (-r,s)\}$, where $r,s > 0$ are chosen so that the segment lies inside the domain $\Omega$.
\end{lemma}

The previous lemma would be applied for a maximal interval $(-r,s)$ so that $\{x_0 + t a(u(x_0)): t \in (-r,s)\}$ is the connected component of $\{x_0 + t a(u(x_0)): t \in \R\} \cap \Omega$ that contains $x_0$.

\begin{proof}
We extend the function $u$ as $u(t,x)$ by making it constant in $t$. Thus, the new function (which we still call $u$) is an entropy solution of \eqref{e-conservationlaw}.

Proposition \ref{p:characteristics} tells us about the existence of backward characteristic curves. In this case, since the function $u$ is continuous, $[\underline u(t,\gamma(t)), \overline u(t,\gamma(t))]$ is a singleton for every value of $t$. The function $u$ and $\gamma'$ will be constant on $\gamma$. Thus, $\gamma$ will be a straight line.

Applying Proposition \ref{p:characteristics}, we see that $u(x_0) = u(0,x_0) = u(t,x_0 + t a(u(x_0)))$ for $t<0$, provided that the segment $\{x_0 + \tau a(u(x_0)): \tau \in (t,0)\} \subset \Omega$.

We need to prove that $u(x_0) = u(t,x_0 + t a(u(x_0)))$ also holds for positive values of $t$. While Proposition \ref{p:characteristics} gives us backward characteristics, it does not give us forward characterisitcs curves.

Let $r>0$ so that $B_r(x_0) \subset \Omega$ and $0 < t \ll r$. Using Proposition \ref{p:characteristics}, we have that for every $x \in B_r(x)$, $u(x - t a(u(x))) = u(x)$. The map $H : x \mapsto x  - t a(u(x))$ is continuous. This map $H$ maps the sphere $\partial B_r(x_0)$ onto some set surrounding $x_0$. 
Since $t \ll r$, this map $H$ has topological degree one around $x_0$. Thus, there is some $x \in B_r$ such that $x - t a(u(x)) = x_0$. Since $u$ is constant on the backward characteristic curve finishing at $x$, we have $u(x) = u(x_0)$ and $x = x_0 + t a(u(x_0))$.

The previous argument shows that we can extend the characteristic curve through $x_0$ forward for a small interval of time. We iterate this procedure until this segment hits $\partial \Omega$.
\end{proof}

\begin{lemma} \label{l:continuous-solutions}
Let $\Omega \subset \R^d$ be an open set and $u : \Omega \to \R$ be a continuous entropy solution to the equation \eqref{e:conservatiolaw-no-t}. Then, there is no entropy dissipation and the equality holds in \eqref{e-entropycondition}.
\end{lemma}

\begin{proof}
Accoding to Lemma \ref{l:characterisitcs-two-sided}, the function $u$ is constant along straight lines with slope $a(u)$.

Let $f$ be the function in \eqref{e:kinetic-formulation}. Note that because $u$ is continuous, the function $f(x,v)$ is well defined at every point $(x,v) \in \Omega \times \R$. We claim $a(v) \cdot \nabla_x f = 0$. This follows from the fact that $f(x+ta(v),v)$ is constant in $t$ for every $x \in \Omega$ and $v \in \R$. Suppose otherwise that for some $x \in \Omega$ and $v \in \R$, the function $t \mapsto f(x+t a(v),v)$ changes values at some $t=t_0$. This necessarily implies that $u(x+t_0 a(v)) = v$ from the definition of $f$ since $u$ is continuous. Then, according to Proposition \ref{p:characteristics}, $u$ must be constant along the same straight line $t \mapsto x + t a(v)$. Thus, $f$ can never change values on that segment.
\end{proof}

\section{Traces and blowup limits}
\label{s:traces}

We conclude this article with a section explaining some consequences of our results in the context of the structure theorems given in \cite{de2003structure}. In that article, the authors prove that $J$ is a rectifiable set. They study blow up limits of the function $u$, continuing some ideas from \cite{vasseur2001}. They prove that $H^{d-1}$ almost everywhere in $J$, the blow up limits are single shocks (a solution consisting of two constants separated by a hyperplane).

Here, the blow up limit at a point $x_0$ is given by
\begin{equation} \label{e:blowup}
 u_\infty(x) = \lim_{r \to 0} u(rx+x_0).
\end{equation}
The limit takes place in $L^1_{loc}$.

We provide the following refinement in the result below. We characterize the two constants in the blow up limit as $\overline u(x)$ and $\underline u(x)$. We prove the blow up limit is uniform away from the interface. Moreover, we show that $u$ has classical nontangential limits $\overline u(x)$ and $\underline u(x)$ at those points in $J$.

\begin{prop} \label{p:traces}
Let us consider a function $a: \R \to \R^d$ for which Assumption \ref{a:genuine-nonlinearity} holds. Let $u$ be an entropy solution to \eqref{e:conservatiolaw-no-t} and $x_0$ be a point in the jump set $J$ so that the blow up limit $u_\infty$ from \eqref{e:blowup} is a single shock. That means that there exists a unit vector $n$, and $u^+ > u^-$, such that
\[  u_\infty(x) = \lim_{r \to 0} u(rx+x_0) = \begin{cases}
u^+ &\text{when } x \cdot n > 0, \\
u^-  &\text{when } x \cdot n < 0.
\end{cases}\]
Then, $u^+ = \overline u(x_0)$ and $u^- = \underline u(x_0)$. The limit is uniform in the set $B_1 \cap \{|x \cdot n|>\delta\}$, for any $\delta>0$. Moreover, for any $\delta>0$,
\[ \lim_{\substack{y \to x_0\\  (y-x_0) \cdot n > \delta |y-x_0| }} u(y) = \overline u(x_0) \qquad \text{and} \qquad \lim_{\substack{y \to x_0\\  (y-x_0) \cdot n < -\delta |y-x_0| }} u(y) = \underline u(x_0).\]
\end{prop}

\begin{proof}
Since $\underline u \leq u \leq \overline u$ almost everywhere, from the semicontinuity of $\underline u$ and $\overline u$, it is clear that $\underline u(x_0) \leq u^- \leq u^+ \leq \overline u(x_0)$. We have to prove the opposite inequalities.

For any $\delta>0$ and $r>0$, let us consider the function $w(x) = (u(rx+x_0) - \overline u(x_0) + \delta)_+$. From Lemma \ref{l:max-of-subsolutions}, we see that $w$ is a subsolution of the equation
\[ a(w(x)+\overline u(x_0)-\delta) \cdot \nabla w(x) \leq 0.\]
Naturally, the function $a(\cdot + \overline u(x_0)-\delta)$ also satisfies Assumption \ref{a:genuine-nonlinearity}. Noticing that $\overline w(0) = \delta$, we apply Theorem \ref{t:degiorgi} and obtain that
\[ \delta \leq \esssup_{B_1} w \leq C \left( \int_{B_2} w \dd x \right)^\gamma \leq C \delta^\gamma |\{ w > 0\} \cap B_2|^\gamma = C \delta^\gamma \left( r^{-d} |\{ u > \overline u(x_0) - \delta\} \cap B_{2r}| \right)^\gamma .\]
Therefore, there must be some constant $c>0$, depending on $\delta$, so that for all $r>0$ small,
\[ |\{ u > \overline u(x_0) - \delta \} \cap B_r | \geq c r^d.\]
Equivalently,
\[ | \{x \in B_1 : u(rx+x_0) > \overline u(x_0) - \delta \}| \geq c.\]
Therefore, passing to the limit in $L^1$ as $r \to 0$, we recover that
\[ | \{ x \in B_1 : u_\infty(x) \geq \overline u(x_0) - \delta \}| \geq c.\]
This allow us to conclude that $u^+ \geq \overline u(x_0)$. Similarly (but upside down), we prove that $u^- \leq \underline u(x_0)$. This finishes the proof of the first statement.

From Theorem \ref{t:degiorgi}, we know that if a sequence of solutions to a conservation law satisfying Assumption \ref{a:genuine-nonlinearity} converges to a constant in $L^1$, then it also converges to a constant uniformly (except perhaps for a set of measure zero). That proves the second statement.

We are left with the third statement about nontangential limits. Let $y_j \to x_0$ so that $(y_j - x_0) \cdot n > \delta |y_j - x_0|$. We want to prove that $\overline u(y_j)$ and $\underline u(y_j)$ both converge to $\overline u(x_0)$.

Let $r_j = |y_j-x_0|$, and $y_j = x_0 + r_j z_j$ with $|z_j|=1$ and $z_j \cdot n > c$. Since $u(x_0 + rx) \to u_\infty$ locally uniformly away from $x \cdot n = 0$ and $r_j \to 0$, we conclude that $\lim_{j \to \infty} \underline u(y_j) = \lim_{j \to \infty} \underline u(x_0 + r_j z_j)  = u^+ = \overline u(x_0)$.

\end{proof}

\bibliographystyle{plain}
\bibliography{continuity}

\begin{thebibliography}{10}

\bibitem{ambrosio2007}
Luigi Ambrosio, Camillo De~Lellis, and Jan Mal\'y.
\newblock On the chain rule for the divergence of {BV}-like vector fields:
  applications, partial results, open problems.
\newblock In {\em Perspectives in nonlinear partial differential equations},
  volume 446 of {\em Contemp. Math.}, pages 31--67. Amer. Math. Soc.,
  Providence, RI, 2007.

\bibitem{bianchini2016}
Stefano Bianchini and Elio Marconi.
\newblock On the concentration of entropy for scalar conservation laws.
\newblock {\em Discrete Contin. Dyn. Syst. Ser. S}, 9(1):73--88, 2016.

\bibitem{bianchini2016structure}
Stefano Bianchini and Elio Marconi.
\newblock On the structure of {$L^\infty$}-entropy solutions to scalar
  conservation laws in one-space dimension.
\newblock {\em Archive for Rational Mechanics and Analysis}, pages 1--53, 2016.

\bibitem{caffarelli2010drift}
Luis~A Caffarelli and Alexis Vasseur.
\newblock Drift diffusion equations with fractional diffusion and the
  quasi-geostrophic equation.
\newblock {\em Annals of Mathematics}, pages 1903--1930, 2010.

\bibitem{chan2014giorgi}
Chi~Hin Chan and Alexis~F Vasseur.
\newblock De giorgi techniques applied to the holder regularity of solutions to
  hamilton-jacobi equations.
\newblock {\em arXiv preprint arXiv:1411.3455}, 2014.

\bibitem{chen1999}
Gui-Qiang Chen and Hermano Frid.
\newblock Decay of entropy solutions of nonlinear conservation laws.
\newblock {\em Arch. Ration. Mech. Anal.}, 146(2):95--127, 1999.

\bibitem{chenperthame2009}
Gui-Qiang Chen and Beno{\^i}t Perthame.
\newblock Large-time behavior of periodic entropy solutions to anisotropic
  degenerate parabolic-hyperbolic equations.
\newblock {\em Proc. Amer. Math. Soc.}, 137(9):3003--3011, 2009.

\bibitem{crippa2008regularizing}
Gianluca Crippa, Felix Otto, and Michael Westdickenberg.
\newblock Regularizing effect of nonlinearity in multidimensional scalar
  conservation laws.
\newblock In {\em Transport equations and multi-D hyperbolic conservation
  laws}, pages 77--128. Springer, 2008.

\bibitem{dafermos1977generalized}
Constantine~M Dafermos.
\newblock Generalized characteristics and the structure of solutions of
  hyperbolic conservation laws.
\newblock {\em Indiana University Mathematics Journal}, 26(6):1097--1119, 1977.

\bibitem{dafermos2006}
Constantine~M. Dafermos.
\newblock Continuous solutions for balance laws.
\newblock {\em Ric. Mat.}, 55(1):79--91, 2006.

\bibitem{dafermos2013}
Constantine~M. Dafermos.
\newblock Long time behavior of periodic solutions to scalar conservation laws
  in several space dimensions.
\newblock {\em SIAM J. Math. Anal.}, 45(4):2064--2070, 2013.

\bibitem{e1957sulla}
Ennio De~Giorgi.
\newblock Sulla differenziabilit\`a e l'analiticit\`a delle estremali degli
  integrali multipli regolari.
\newblock {\em Mem. Accad. Sci. Torino. Cl. Sci. Fis. Mat. Nat. (3)}, 3:25--43,
  1957.

\bibitem{de2003structure}
Camillo De~Lellis, Felix Otto, and Michael Westdickenberg.
\newblock Structure of entropy solutions for multi-dimensional scalar
  conservation laws.
\newblock {\em Archive for rational mechanics and analysis}, 170(2):137--184,
  2003.

\bibitem{deLellisRiviere2003}
Camillo De~Lellis and Tristan Rivi\`ere.
\newblock The rectifiability of entropy measures in one space dimension.
\newblock {\em J. Math. Pures Appl. (9)}, 82(10):1343--1367, 2003.

\bibitem{de2003optimality}
Camillo De~Lellis and Michael Westdickenberg.
\newblock On the optimality of velocity averaging lemmas.
\newblock In {\em Annales de l'Institut Henri Poincare (C) Non Linear
  Analysis}, volume~20, pages 1075--1085. Elsevier, 2003.

\bibitem{Debussche2009}
Arnaud Debussche and J.~Vovelle.
\newblock Long-time behavior in scalar conservation laws.
\newblock {\em Differential Integral Equations}, 22(3-4):225--238, 2009.

\bibitem{gess2017long}
Benjamin Gess and Panagiotis~E Souganidis.
\newblock Long-time behavior, invariant measures, and regularizing effects for
  stochastic scalar conservation laws.
\newblock {\em Communications on Pure and Applied Mathematics},
  70(8):1562--1597, 2017.

\bibitem{golse2016harnack}
Fran\c{c}ois Golse, Cyril Imbert, Cl{\'e}ment Mouhot, and Alexis Vasseur.
\newblock Harnack inequality for kinetic fokker-planck equations with rough
  coefficients and application to the {L}andau equation.
\newblock {\em arXiv preprint arXiv:1607.08068}, 2016.

\bibitem{golse1988}
Fran\c{c}ois Golse, Pierre-Louis Lions, Beno\^it Perthame, and R\'emi Sentis.
\newblock Regularity of the moments of the solution of a transport equation.
\newblock {\em J. Funct. Anal.}, 76(1):110--125, 1988.

\bibitem{imbert2016weak}
Cyril Imbert and Luis Silvestre.
\newblock The weak {H}arnack inequality for the boltzmann equation without
  cut-off.
\newblock {\em Journal of the European Mathematical Society}, In press.

\bibitem{kruvzkov1970first}
Stanislav~N Kru{\v{z}}kov.
\newblock First order quasilinear equations in several independent variables.
\newblock {\em Mathematics of the USSR-Sbornik}, 10(2):217, 1970.

\bibitem{lax1957}
Peter~D. Lax.
\newblock Hyperbolic systems of conservation laws. {II}.
\newblock {\em Comm. Pure Appl. Math.}, 10:537--566, 1957.

\bibitem{lions-video}
Pierre~Louis Lions.
\newblock Coll\`ege de france, s\'eminaire: Du nouveau sur les lois de
  conservation scalaires ?
\newblock
  \url{https://www.college-de-france.fr/site/pierre-louis-lions/seminar-2016-11-18-11h15.htm},
  2016.
\newblock Online. Accessed on 2018/02/02.

\bibitem{lions1994kinetic}
Pierre-Louis Lions, Beno{\i}t Perthame, and Eitan Tadmor.
\newblock A kinetic formulation of multidimensional scalar conservation laws
  and related equations.
\newblock {\em Journal of the American Mathematical Society}, 7(1):169--191,
  1994.

\bibitem{panov2013}
E.~Yu. Panov.
\newblock On decay of periodic entropy solutions to a scalar conservation law.
\newblock {\em Ann. Inst. H. Poincar\'e Anal. Non Lin\'eaire}, 30(6):997--1007,
  2013.

\bibitem{perthame2002kinetic}
Beno{\^\i}t Perthame.
\newblock {\em Kinetic formulation of conservation laws}, volume~21.
\newblock Oxford University Press, 2002.

\bibitem{silvestre2016transport}
Luis Silvestre and Vlad Vicol.
\newblock On a transport equation with nonlocal drift.
\newblock {\em Transactions of the American Mathematical Society},
  368(9):6159--6188, 2016.

\bibitem{tadmor2007velocity}
Eitan Tadmor and Terence Tao.
\newblock Velocity averaging, kinetic formulations, and regularizing effects in
  quasi-linear pdes.
\newblock {\em Communications on pure and applied mathematics},
  60(10):1488--1521, 2007.

\bibitem{vasseur2001}
Alexis Vasseur.
\newblock Strong traces for solutions of multidimensional scalar conservation
  laws.
\newblock {\em Arch. Ration. Mech. Anal.}, 160(3):181--193, 2001.

\end{thebibliography}

\end{document}